\documentclass[a4paper, 11pt, titlepage, fleqn]{article}
\usepackage{times}
\usepackage{amsthm}
\newtheorem{thm}{Theorem}
\newtheorem{corollary}{Corollary}
\newtheorem{conj}{Conjecture}

\newtheorem{proposition}{Proposition}

\usepackage{amssymb,amsmath}

\newcommand{\Aut}{\operatorname{Aut}}
\newcommand{\Jac}{\operatorname{Jac}}

\begin{document}
	\baselineskip=16.3pt
	\parskip=14pt
	\begin{center}
		\section*{Divisibility of L-Polynomials for a Family of Curves}

		{\large
		Ivan  Blanco--Chac\'on\\
		School of Mathematics and Statistics, University College Dublin, Ireland \\ \bigskip
		Robin Chapman\\
Department of Mathematics, University of Exeter, Exeter, EX4~4QE, UK \\ \bigskip
Stiof\'ain Fordham\\
School of Mathematics and Statistics, University College Dublin, Ireland \\ \bigskip
Gary McGuire\footnote{email gary.mcguire@ucd.ie. This paper appeared in the proceedings of Fq12, the 12th International Conference on Finite Fields and their Applications, 2015.} \\
School of Mathematics and Statistics, University College Dublin, Ireland}

	\end{center}
	\subsection*{Abstract}
	We consider the question of
when the L-polynomial of
one curve divides the L-polynomial of another curve.
A theorem of Tate gives an answer in terms of jacobians.
We consider the question in terms of the curves.
The last author gave an invited talk at the 12th International Conference on Finite Fields and Their Applications on this topic, and stated two conjectures.
In this article we prove one of those conjectures.

\section{Introduction} \label{chapman:intro}

Let $p$ be a prime and $q=p^f$ where $f$ is a positive integer and let $\mathbb{F}_q$ denote the finite field of order $q$. Let $X$ be a smooth projective variety over $\mathbb{F}_q$, of dimension $d$. Let $\overline X=X(\overline{\mathbb{F}}_q)$ be the corresponding variety over the algebraic closure of $\mathbb{F}_q$ and let $F\colon \overline X\rightarrow \overline X$ be the Frobenius morphism. The zeta function $\mathrm{Z}_X(t)$ of $X$ is defined by
\[
\log \mathrm{Z}_X(t)= \sum_{m\ge 1} \frac{t^m}{m}N_m,
\]
where $N_m$ is the cardinality of the set $X(\mathbb{F}_{q^m})$: the points of $X$ with values in $\mathbb{F}_{q^m}$. Via the Weil conjectures (proved by Weil, Dwork, Grothendieck and others), one knows that $\mathrm{Z}_X(t)$ is a rational function and may be written in the form
\begin{equation}
\mathrm{Z}_X(t)=\frac{P_1(t)\cdots P_{2d-1}(t)}{P_0(t)\cdots P_{2d}(t)},\label{chapman:weil-conj}
\end{equation}
where each of the $P_i(t)=\det (1-F^\ast t; H^i(\overline X, \mathbb{Q}_\ell))$ are polynomials with coefficients in $\mathbb{Z}$ , where $H^i(\overline X , \mathbb{Q}_\ell)$ is the $i$th $\ell$-adic cohomology ($\ell \ne p$) of $\overline{X}$ with coefficients in $\mathbb{Q}_\ell$ and $F^\ast$ is the map on cohomology induced by $F$.

In the case that $X=C$ is a curve then the zeta function of $C$ has the form
\[
\mathrm{Z}_C(t)=\frac{\mathrm{L}_C(t)}{(1-t)(1-qt)},
\]
and the numerator $\mathrm{L}_C(t)= P_1(t)$ is called the L-polynomial of $C$.

We wish to consider the question of divisibility of L-polynomials.
In previous papers \cite{chapman:AM}, \cite{chapman:AMR},
we have studied conditions under which the L-polynomial of
one curve divides the L-polynomial of another curve.
In this article we discuss two divisibility conjectures for specific families of curves,
and prove one of them.

\section{Two Families of Curves}

A hyperelliptic curve $X$ of genus $g>1$ over 
$\mathbb{F}_q$ is the projective non-singular model of the affine curve
\[
y^2+Q(x)y=P(x),\qquad P(x),Q(x)\in \mathbb{F}_q [x],
\]
where
\[
2g+1\le \max\{2\deg Q(x),\deg P(x)\}\le 2g+2.
\]

\subsection{The $C_k$ Family}

For a positive integer $k$, define the curve $C_k$ over $\mathbb{F}_2$ to be the projective 
non-singular model of the curve with affine equation
\[
y^2+y=x^{2^k+1}+x.
\]
The genus of $C_k$ is $2^{k-1}$, the affine model of $C_1$ is smooth everywhere, and the affine model of $C_k$ for $k>1$ has one singular point at $\infty$.
\begin{conj}\label{chapman:conjck}
The $\mathrm{L}$-polynomial of $C_k$ is divisible by the $\mathrm{L}$-polynomial of $C_1$.
\end{conj}

The first six L-polynomials over $\mathbb{F}_2$, 
computed and factored into irreducible factors over $\mathbb{Z}$ using MAGMA
\cite{chapman:magma} are

\begin{align*}
C_1&\colon 2t^2 + 2t + 1\\
C_2&\colon (2t^2 + 2t + 1)(2t^2 + 1)\\
C_3&\colon ( 2t^2 + 2t + 1)(2t^2 - 2t + 1)(4t^4 + 4t^3 + 2t^2 + 2t + 1)\\
C_4&\colon (2t^2 + 2t + 1)^2 (2t^2 - 2t + 1)(2t^2 + 1)(16t^8 + 1)\\
C_5&\colon (2t^2 + 2t + 1)^2  (2t^2 - 2t + 1)^2(16t^8 - 16t^7 + 8t^6 - 4t^4 + 2t^2 - 2t + 1)\\
&\qquad \times (16t^8 + 16t^7 + 8t^6 - 4t^4 + 2t^2 + 2t + 1)^2\\
C_6&\colon 
 (2t^2 - 1)^2 (2t^2 + 1)^4 (4t^4 - 2t^2 + 1)^3  (4t^4 + 2t^2 + 1)^2\\
&\qquad \times (2t^2 - 2t + 1)^3
 (2t^2 + 2t + 1)^3
    (4t^4 - 4t^3 + 2t^2 - 2t + 1)^2\\
    &\qquad \times (4t^4 + 4t^3 + 2t^2 + 2t + 1)^3.
\end{align*}

\subsection{The $E_k$ Family}

\indent For a positive integer $k$, define the curve $E_k$ over $\mathbb{F}_2$ to be the 
projective non-singular model of the curve with affine model
\[
y^2+xy=x^{2^k+3}+x.
\]
The genus of $E_k$ is $2^{k-1}+1$ and similar to above, the affine model of $E_1$ is smooth everywhere, and the affine model of $E_k$ for $k>1$ has one singular point at $\infty$.

\begin{conj}\label{chapman:conjdk}
The $\mathrm{L}$-polynomial of $E_k$ is divisible by the $\mathrm{L}$-polynomial of $E_1$.
\end{conj}

In an invited talk at the Fq12 conference, the last author spoke about this topic
and stated these two conjectures. Conjecture \ref{chapman:conjdk} is proposed and
discussed in Ahmadi et al.\ \cite{chapman:AMR}.
In this paper we will prove Conjecture \ref{chapman:conjck}.

\section{Other Approaches}

Here we discuss three possible approaches to proving the conjectures.
The first two do not seem to work for Conjecture \ref{chapman:conjck},
but the third method does work as we will show in this paper. 
None of these methods appear to work for proving Conjecture \ref{chapman:conjdk}.

\subsection{Number of Rational Points}

The following theorem was proved in \cite{chapman:AM}.

\begin{thm}[Ahmadi--McGuire]\label{chapman:AMBCC}
Let $C(\mathbb{F}_q)$ and $D(\mathbb{F}_q)$ be smooth projective curves such that
\begin{enumerate}
\item $C(\mathbb{F}_q)$ and $D(\mathbb{F}_q)$ have the same number of points over infinitely many extensions  of $\mathbb{F}_q$. 
\item The $\mathrm{L}$-polynomial of $C$
over $\mathbb{F}_{q^k}$ has no repeated roots, for all $k\geq 1$.
\end{enumerate}
Then there exists a positive integer $s$ such that 
the $\mathrm{L}$-polynomial of $D(\mathbb{F}_{q^s})$ is divisible by the $\mathrm{L}$-polynomial of $C(\mathbb{F}_{q^s})$. 
\end{thm}

The first hypothesis holds for the curves $C_k$, 
but the second hypothesis does not. 
Thus we cannot use Theorem \ref{chapman:AMBCC} to prove Conjecture \ref{chapman:conjck}.
To see that the first hypothesis holds, we use
the following theorem proved by Lahtonen--McGuire--Ward \cite{chapman:LMW}.

\begin{thm}\label{chapman:lmward}
Let $K=\mathbb{F}_{2^n}$ where $n$ is a non-negative odd integer. 
Let
\[
Q(x)=\mathrm{Tr}(x^{2^k+1}+x^{2^j+1}),\qquad\text{ for }0\leq j<k.
\]
Then if $\gcd(k\pm j,n)=1$, then the number of zeros of $Q$ in $K$ is
\[
2^{n-1}+ \left(\frac{2}{n}\right) 2^{(n-1)/2},
\] 
where $(\frac{2}{n})$ is the Jacobi symbol.  
\end{thm}

If we put $j=0$ in Theorem \ref{chapman:lmward} then $Q(x)=\mathrm{Tr}(x^{2^k+1}+x)$.
It follows that $C_1$ and $C_k$ have the same number of rational 
points over $\mathbb{F}_{2^m}$ for any $m$ with $\gcd(k,m)=1$.
Therefore the first hypothesis of Theorem \ref{chapman:AMBCC} holds.

The $C_k$ curves are supersingular so the L-polynomial of $C_1$ 
(which is $2t^2+2t+1$) has
repeated roots over some extensions of $\mathbb{F}_2$, something that can also be seen directly.
Therefore the second hypothesis of Theorem \ref{chapman:AMBCC}  does not hold.

\bigskip

We remark that the second hypothesis of Theorem \ref{chapman:AMBCC}  \emph{does} hold
for the $E_k$ curves, see \cite{chapman:AMR}.
However, we are unable to prove that the first hypothesis holds, although it is conjectured
that it does. 

\bigskip

A similar but different theorem was proved in \cite{chapman:AMR}.

\begin{thm}[Ahmadi--McGuire--Rojas-Le\'on]\label{chapman:AMRL} 
Let $C$ and $D$ be two smooth projective curves over $\mathbb{F}_q$.
Assume there exists a positive integer $k>1$ such that
\begin{enumerate}
\item $\#C(\mathbb{F}_{q^m})=\#D(\mathbb{F}_{q^m})$ for every $m$ that is not divisible by $k$, and
\item the $k$-th powers of the roots of $\mathrm{L}_C(t)$ are all distinct. 
\end{enumerate}
Then  $\mathrm{L}_D(t)=q(t^k)\ \mathrm{L}_C(t)$ for some polynomial $q(t)$ in $\mathbb{Z}[t]$. 
\end{thm}

We cannot use Theorem \ref{chapman:AMRL} to prove Conjecture \ref{chapman:conjck},
because the first hypothesis does not hold. 
It is not true that $C_1$ and $C_k$ have the same number of rational 
points over $\mathbb{F}_{2^m}$ for any $m$ not divisible by $k$ (or another integer).
This can be seen by looking at small examples using a computer algebra package.

It is interesting to compare the first hypothesis in Theorem \ref{chapman:AMRL} with
the first  hypothesis in Theorem \ref{chapman:AMBCC} (which does hold for $C_1$ and $C_k$).

\subsection{The Kani--Rosen theorem}

Let $X$ be an affine variety over a field $k$ with coordinate ring $A$. Given an action of an algebraic group $G$ on $X$, one may construct a so-called quotient variety $X/G$ given by $\mathrm{Spec}(A^G)$ where $A^G$ denotes the ring of invariants of $A$ with the induced action of $G$. If furthermore, $G$ is reductive then $A^G$ is finitely generated so $X/G$ is also an affine variety ($A^G$ will be reduced if $A$ is).

Let $G$ be a finite subgroup of the automorphism group of a curve $C$ and let $\Jac(C)$ denote the Jacobian of $C$.
The Kani--Rosen theorem \cite[thm.\ B]{chapman:KR} concerns isogenies and idempotents in the rational group algebra 
$\mathbb{Q}[G]$ and is useful is proving divisibility relations between L-polynomials.

\begin{thm}[Kani--Rosen]
Let $G \subseteq \mathrm{Aut}(C)$ be a (finite) subgroup such that $G=H_1 \cup H_2 \cup \ldots \cup H_r$ where the subgroups $H_i \subseteq G$ satisfy $H_i\cap H_j=\{1\}$ when $i\ne j$. Then there is an isogeny relation
\[
\Jac(C)^{r-1}\times \Jac(C/G)^g \cong \Jac(C/H_1)^{h_1}\times \ldots \times \Jac(C/H_r)^{h_r},
\]
where $g=|G|$ and $h_i=|H_i|$.
\end{thm}

For any subgroup $H$ of $G$ there is an idempotent
\[
\varepsilon_H = \frac{1}{|H|}\sum_{h\in H} h.
\]
If $G$ is the Klein 4-group with subgroups $H_1, H_2, H_3$, we have the 
idempotent relation
\[
\varepsilon_1 + 2\varepsilon_G=\varepsilon_{H_1}+\varepsilon_{H_2}+\varepsilon_{H_3}.
\]
Applying the Kani--Rosen theorem we get an isogeny
\[
{\rm Jac}(C) \times {\rm Jac}({C/G})^2 \sim 
{\rm Jac}({C/H_1})\times {\rm Jac}({C/H_2})\times {\rm Jac}({C/H_3}).
\]

In order to apply this isogeny to $C_k$, we need two involutions in the 
automorphism group  of $C_k$.  
We want involutions that are defined over $\mathbb{F}_2$.
One is the hyperelliptic involution
\[
\iota: (x,y)\mapsto (x,y+1)
\]
and the other is the map from \cite{chapman:GV}
\[
\phi: (x,y)\mapsto (x+1,y+B(x))
\]
where $B(x)=x+x^2+x^4+x^8+\cdots +x^{2^{k-1}}$. Then
\[
\iota \circ \phi = \phi \circ \iota : (x,y) \mapsto (x+1,y+1+B(x)).
\]
Note that $\phi$ is an involution if and only if $B(1)=0$
if and only if $k$ is even.
When $k$ is odd, $\phi$ has order 4 and $\phi^2=\iota$. 
When $k$ is even, $\phi$ and $\iota$ together generate a Klein 4-group in $\Aut(C_k)$.
In fact we have the following.

\begin{proposition}
If $k$ is odd then there are no non-hyperelliptic involutions on $C_k$ of the form $(x,y)\mapsto (x+1,y+B(x))$ where $B(x)$ is a linearised polynomial $B(x)=\sum_{i\ge 0} a_ix^{2^i}$ with $a_i\in \mathbb{F}_2$.
\end{proposition}

\begin{proof}
Such a $B(x)$ must satisfy $B(1)=0$ and $B(x)^2+B(x)=x^{2^k}+x$. The resulting conditions thus imposed on the coefficients $a_i$ mean that $B(x)=\sum_{i=0}^{k-1} x^{2^i}$ but then $B(1)\ne 0$ if $k$ is odd.
\end{proof}

Remark:
It follows now from van der Geer and van der Vlugt \cite{chapman:GV} that this exhausts the subgroup of $\Aut_{\mathbb{F}_2}(C_k)$ fixing the branch points of $C_k\rightarrow \mathbb{P}^1$.

Therefore, the first problem in using the Kani--Rosen theorem 
to prove Conjecture \ref{chapman:conjck}
is that we only have
the appropriate automorphism group for $k$ even. It therefore
appears that for $k$ odd, one cannot use the Kani--Rosen theorem to prove the conjecture,
at least not directly.

\subsection{Kleiman--Serre}

The following theorem is well-known in the area.

\begin{thm}\label{chapman:KleimanSerre} 
(Kleiman--Serre)
If there is a surjective morphism of curves $C \longrightarrow C'$ that is defined over $\mathbb{F}_q$
then $\mathrm{L}_{C'}(t)$ divides $\mathrm{L}_C(t)$.
\end{thm}

\begin{proof} (Sketch)
Given a surjective morphism $f\colon C\rightarrow C'$ one obtains an induced map $f^\ast$ on the \'etale cohomology groups that is injective (Kleiman \cite[prop.\ 1.2.4]{chapman:Kleiman}). Given the interpretation of the polynomials $P_i(t)$ described in the introduction (equation \ref{chapman:weil-conj}) as determinants via the Weil conjectures, the result follows.
\end{proof}

We will use this result in the next section to prove Conjecture \ref{chapman:conjck}.

\section{Proof of Conjecture \ref{chapman:conjck}}

We prove Conjecture \ref{chapman:conjck} using Theorem \ref{chapman:KleimanSerre}.
In fact we will prove something more general: that there is a map
from $C_{k}$ to $C_l$ for any integer $l$ dividing $k$.
Putting $l=1$ proves Conjecture \ref{chapman:conjck}.

Before we construct the morphism we consider a simpler case as motivation. Let $A_k$ denote the smooth projective model of the affine curve defined over $\mathbb{F}_2$
\[
y^2+y=x^{2^k}+x.
\]
One can easily verify that the map
\[
(x,y)\mapsto (\mathrm{Tr}_{nk/k}(x),y), \qquad \text{ for }\mathrm{Tr}_{nk/k}(x)=x+\sum_{i=1}^{n-1}x^{2^{ik}},
\]
is a morphism  $A_{nk}\rightarrow A_k$ for $n,k$ positive integers and $n>1$.

The similarity between the curves $C_k$ and $A_k$ is apparent, however the morphism above differs quite radically from the one to be described below.

\begin{thm}\label{chapman:ckdiv}
Let $k>l$ be integers with $l$ dividing $k$.
Then there is a non-constant morphism $C_k\longrightarrow C_l$ defined over $\mathbb{F}_2$.
\end{thm}

\begin{proof}
Write $k=lr$ and set $q=2^l$.
We claim that there is a morphism of the form $x\mapsto f(x)$,
$y\mapsto y+g(x)$ from $C_k$ to $C_l$ where $f$ and $g$ are polynomials.
For this to be the case it suffices that
\begin{equation}\label{chapman:covering-condition}
f(x)^{q+1}+f(x)=x^{q^r+1}+x+g(x)^2+g(x).
\end{equation}
Let us take
\[
f(x)=\sum_{j=0}^{r-1}x^{q^j},
\]
and
\[
g(x)=\sum_{j=1}^{r-1}x^{q^j}+\sum_{0\le i<j\le r-1}\sum_{s=0}^{l-1}x^{2^l(q^i+q^j)}.
\]
Then
\begin{eqnarray*}
f(x)+f(x)^{q+1}
&=&\sum_{j=0}^{r-1}x^{q^j}
+\left(x+\sum_{j=1}^{r-1}x^{q^j}\right)
\left(x^{q^r}+\sum_{j=1}^{r-1}x^{q^j}\right)\\
&=&\sum_{j=1}^{r-1}x^{q^j}
+x+x^{1+q^r}+\sum_{j=1}^{r-1}x^{1+q^j}+\sum_{j=1}^{r-1}x^{q^r+q^j}
+\sum_{j=1}^{r-1}x^{2q^j},
\end{eqnarray*}
and
\begin{eqnarray*}
g(x)^2+g(x)
&=&\sum_{j=1}^{r-1}x^{q^j}+\sum_{j=1}^{r-1}x^{2q^j}
+\sum_{0\le i<j\le r-1}(x^{q^i+q^j}+x^{q(q^i+q^j)})\\
&=&\sum_{j=1}^{r-1}x^{q^j}+\sum_{j=1}^{r-1}x^{2q^j}
+\sum_{0\le i<j\le r-1}x^{q^i+q^j}
+\sum_{1\le i<j\le r}x^{q^i+q^j}\\
&=&\sum_{j=1}^{r-1}x^{q^j}+\sum_{j=1}^{r-1}x^{2q^j}
+\sum_{j=1}^{r-1}x^{1+q^j}
+\sum_{i=1}^{r-1}x^{q^i+q^r}.
\end{eqnarray*}
Subtracting these gives~(\ref{chapman:covering-condition}).
\end{proof}

\begin{corollary}
Conjecture \ref{chapman:conjck} is true.
\end{corollary}

The Corollary follows from Theorem \ref{chapman:KleimanSerre}
and Theorem \ref{chapman:ckdiv}.

\bigskip

We used Theorem \ref{chapman:KleimanSerre} to prove Conjecture \ref{chapman:conjck}.
We remark that Theorem \ref{chapman:KleimanSerre}
cannot be used to prove Conjecture \ref{chapman:conjdk},
because it is shown in \cite{chapman:AMR} that there is no morphism
$E_2 \longrightarrow E_1$.
Thus a proof of Conjecture \ref{chapman:conjdk} will probably use different methods.

\bigskip

As a final remark, we point out where the argument of Theorem \ref{chapman:ckdiv} breaks down in
odd characteristic for the analogous curves
\[
C_k^{(p)}:\qquad y^p-y=x^{p^k+1}+x,
\]
where $p$ is an odd prime.
In the case $k=2$, $l=1$, in order
to give a morphism of the form $(x,y)\mapsto (f(x),y+g(x))$ from
$C_2^{(p)}$ to $C_1^{(p)}$
we need to find polynomials $f$ and $g$
with
\[
f(x)^{p+1}+f(x)=x^{p^2+1}+x+g(x)^p-g(x).
\]
If we take $f(x)=x+x^p$ by analogy with Theorem~\ref{chapman:ckdiv}, then we require
\[
x^{p^2+p}+x^{2p}+x^{p+1}+x^p=g(x)^p-g(x),
\]
but this is insoluble for polynomial $g$ unless $p=2$.

Notwithstanding the above, the analogous conjecture for odd $p$ does appear to be true based on computations for small $k,p$.
\begin{conj}
Let $p$ be an odd prime. Then the L-polynomial of $C_1^{(p)}$ divides the L-polynomial of $C_k^{(p)}$.
\end{conj}

\paragraph{Acknowledgment}
I.B.--C., S.F.\ and G.M.\ are supported by Science Foundation Ireland grant 13/IA/1914 and are members of the Computational and Adaptive Systems Laboratory (CASL) in University College Dublin. I.B.--C.\ is a member of the MICINN project MTM2010-17389. S.F.\ is partially supported by an Irish Department of Education scholarship.

\end{document}